\documentclass[12pt]{amsart}

\usepackage{amsfonts,amsmath}
\newtheorem{thm}{Theorem}[section]
\newtheorem{lemma}[thm]{Lemma}
\newtheorem{prop}[thm]{Proposition}
\newtheorem{cor}[thm]{Corollary}

\newcommand{\zz}{\begin{equation}}
\newcommand{\z}{\end{equation}}

\newcommand{\E}{\epsilon}
\newcommand{\IE}{\mathbb{E}}
\newcommand{\IP}{\mathbb{P}}
\newcommand{\R}{\mathbb{R}}
\newcommand{\Z}{\mathbb{Z}}
\begin{document}

\title{Uniqueness of solutions of stochastic differential equations}

\author{A. M. Davie}

\address{School of Mathematics, University of Edinburgh, King's Buildings,
Mayfield Road, Edinburgh EH9 3JZ, UK}


%\date{January }

%\keywords{}

\begin{abstract}
We consider the stochastic differential equation
\[dx(t)=dW(t)+f(t,x(t))dt,\ \ \ \ \ \ \ \ \ \ x(0)=x_0\]
for $t\geq0$, where $x(t)\in\R^d$, $W$ is a standard $d$-dimensional
Brownian motion, and $f$ is a bounded Borel function from $[0,\infty)
\times\R^d\rightarrow\R^d$ to $\R^d$. We show that, for almost all Brownian
paths $W(t)$, there is a unique $x(t)$ satisfying this equation.
\end{abstract}

\maketitle

\section{Introduction}\label{int}

In this paper we consider the stochastic differential equation
\[dx(t)=dW(t)+f(t,x(t))dt,\ \ \ \ \ \ \ \ \ \ x(0)=x_0\]
for $t\geq0$, where $x(t)\in\R^d$, $W$ is a standard $d$-dimensional
Brownian motion, and $f$ is a bounded Borel function from $[0,\infty)
\times\R^d\rightarrow\R^d$ to $\R^d$. Without loss of generality we
suppose $x_0=0$ and then we can write the equation as
\zz\label{eq1} x(t)=W(t)+\int_0^tf(s,x(s))ds,\ \ \ \ \ t\geq0\z
It follows from a theorem of Veretennikov \cite{ver} that (\ref{eq1}) has a
unique strong solution, i.e. there is a unique process $x(t)$, adapted to
the filtration of the Brownian motion, satisfying (\ref{eq1}). Veretennikov
in fact proved this for a more general equation. Here we consider a different
question, posed by N. V. Krylov \cite{ig}: we choose a Brownian path $W$ and
ask whether (\ref{eq1}) has a unique solution for that particular path. The
main result of this paper is the following affirmative answer:

\begin{thm}\label{mth} For almost every Brownian path $W$, there is a unique
continuous $x:[0,\infty)\rightarrow\R^d$ satisfying (\ref{eq1}).
\end{thm}

This theorem can also be regarded as a uniqueness theorem for a random
ODE: writing $x(t)=W(t)+u(t)$, the theorem states that for almost all
choices of $W$, the differential equation $\frac{du}{dt}=f(t,W(t)+u(t))$
with $u(0)=0$ has a unique solution.

In Section \ref{app}, we give an application of this theorem to convergence of
numerical approximations to (\ref{eq1}).\vspace{.2cm}\\
{\bf Idea of proof of theorem.}
The theorem is trivial when $f$ is Lipschitz in $x$, and the idea of the
proof is essentially to find some substitute for a Lipschitz condition.
The proof splits into two parts, the first (section \ref{bes}) being the
derivation of an estimate which acts as a substitute for the Lipschitz
condition, and the second (section \ref{pf}) being the application of this
estimate to prove the theorem. We start with a reduction to a slightly
simpler problem.\vspace{.2cm}\\
{\bf A reduction.} It will be convenient to suppose $|f(t,x)|\leq1$
everywhere, which we can by scaling. Then it will suffice to prove
uniqueness of a solution on [0,1], as we can then repeat to get
uniqueness on [1,2] and so on.

So we work on [0,1], let $X$ be the space of continuous functions
$x:[0,1]\rightarrow\R^d$ with $x(0)=0$, and let $P_W$ be the law of
$\R^d$-valued Brownian motion on [0,1], which can be regarded as a
probability measure on $X$. Now we apply the Girsanov theorem (see
\cite{iw}): define $\phi(x)=\exp\{\int_0^1f(t,x(t))dx(t)-\frac12
\int_0^1f(t,x(t))^2dt\}$, which is well-defined for $P_W$ almost all
 $x\in X$, and define a measure $\mu$ on $X$ by $d\mu=\phi dP_W$. Then
if $x\in X$ is chosen at random with law $\mu$, the path $W\in X$
defined by 
\zz\label{eq2} W(t)=x(t)-\int_0^tf(s,x(s))ds\end{equation}
is a Brownian motion, i.e. $W$ has law $P_W$.

For a particular choice of $x$, and with $W$ defined by (\ref{eq2}), $x$ will
be the unique solution of (\ref{eq1}) provided the only solution of
\zz\label{eq3} u(t)=\int_0^t\{f(s,x(s)+u(s))-f(s,x(s))\}ds\end{equation}
in $X$ is $u=0$. So, to prove the theorem it suffices to show that, for
$\mu$-a.a. $x$, (\ref{eq3}) has no non-trivial solution, since for such $x$,
with $W$ defined by (\ref{eq2}) no other $x$ can satisfy (\ref{eq2}).

But $\mu$ is absolutely continuous w.r.t. $P_W$, so it suffices to show
that, for $P_W$-a.a. $x$, (\ref{eq3}) has no non-trivial solution. In other
words, it suffices to show that, if $W$ is a Brownian motion then with
probability 1 there is no non-trivial solution $u\in X$ of
\zz\label{eq4} u(t)=\int_0^t\{f(s,W(s)+u(s))-f(s,W(s))\}ds\end{equation}
We prove this in section \ref{pf}.\vspace{.2cm}\\
{\bf Remark.} Our proof does not make use of the existence of a strong solution.
It is tempting to try to prove the theorem by measure-theoretic arguments
based on the strong solution and Girsanov's theorem. Define $T: 
X\rightarrow X$ by
\[Tx(t)=x(t)-\int_0^tf(s,x(s))ds\]
The strong solution gives a measurable map $S: E\rightarrow F$ where $E$ and
$F$ are Borel subsets of $X$ with $P_W(E)=P_W(F)=1$, such that $T\circ S$ is
the identity on $E$, and $F$ is the range of $S$. It follows that $T$ is
(1-1) on $F$ and for any $W\in E$ there is a unique solution of (\ref{eq1})
in $F$. But we need a solution which is unique in $X$ and to achieve this we
need to show that $T(X\backslash F)$ is a $P_W$-null set, and this seems
to be a significant obstacle.

Our proof is quite complicated and it seems reasonable to hope that it
can be simplified. In particularly one might expect a simpler proof of
Proposition \ref{mp}. This seems to be nontrivial even for $p=2$. The bound for
$p=2$ follows from the first part of Lemma \ref{bl2} (with $t_0=0$ and
$r=0$) and I do not know an essentially simpler proof.

In one dimension, in the case when $f(t,x)$ depends only on $x$, a different and shorter proof of Theorem \ref{mth} can
be given, using local time, but it is not clear how to extend it to $d>1$.

\section{The basic estimate}\label{bes}

This section is devoted to the proof of the following:
\begin{prop}\label{be} Let $g$ be a Borel function on $[0,1]\times\R^d$ with
$|g(s,z)|\leq1$ everywhere. For any even positive integer $p$ and
$x\in\R^d$, we have
\[\IE\left(\int_0^1\{g(t,W(t)+x)-g(t,W(t))\}dt\right)^p\leq C^p(p/2)!|x|^p\]
where $C$ is an absolute constant, $|x|$ denotes the usual Euclidean norm
and $W(t)$ is a standard $d$-dimensional Brownian motion with $W(0)=0$,
\end{prop}

This will be deduced from the following one-dimensional version:

\begin{prop}\label{mp} Let $g$ be a compactly supported smooth function on $[0,1]\times\R$ with
$|g(s,z)|\leq1$ everywhere and $g'$ bounded (where the prime denotes
differentiation w.r.t. the second variable). For any even positive integer
$p$, we have
\[\IE\left(\int_0^1g'(t,W(t))dt\right)^p\leq C^p(p/2)!\]
where $C$ is an absolute constant, and here $W(t)$ is one-dimensional
Brownian motion with $W(0)=0$.
\end{prop}

\begin{proof} We start by observing that the LHS can be written as
\[p!\int_{0<t_1<\cdots<t_p<1}\IE\prod_{j=1}^pg'(t_j,W(t_j))dt_1\cdots dt_p\]
and using the joint distribution of $W(t_1),\cdots,W(t_p)$ this can be
expressed as
\[p!\int_{0<t_1<\cdots<t_p<1}\int_{\R^p}\prod_{j=1}^p\{g'(t_j,z_j)
E(t_j-t_{j-1},z_j-z_{j-1})\}dz_1\cdots dz_pdt_1\cdots dt_p\]
where $E(t,z)=(2\pi t)^{-1/2}e^{-z^2/2t}$ and here $t_0=0$, $z_0=0$.

We introduce the notation
\[J_k(t_0,z_0)=\int_{t_0<t_1<\cdots<t_k<1}\int_{\R^k}\prod_{j=1}^k\{
g'(t_j,z_j)E(t_j-t_{j-1},z_j-z_{j-1})\}dz_1\cdots dz_kdt_1\cdots dt_k\]
and we shall show that $J_p(0,0)\leq C^p/\Gamma(\frac p2+1)$; Proposition
\ref{mp} will then follow since\\$p!\leq2^p((p/2)!)^2$.

In order to estimate $J_k$ we use integration by parts to shift the
derivatives to the exponential terms. We introduce some notation to handle
the resulting terms - we define $B(t,z)=E'(t,z)$ and $D(t,z)=E''(t,z)$
(where again primes denote differentiation w.r.t. the second variable). 

If $S=S_1\cdots S_k$ is a word in the alphabet $\{E,B,D\}$ then we define
\[I_S(t_0,z_0)=\int_{t_0<t_1<\cdots<t_k<1}\int_{\R^d}\prod_{j=1}^k
\{g(t_j,z_j)S_j(t_j-t_{j-1},z_j-z_{j-1})\}dz_1\cdots dz_kdt_1\cdots dt_k\]

In fact, only certain words in $\{E,B,D\}$ will be required: we say a
word is {\em allowed} if, when all $B$'s are removed from the word, a
word of the form $(ED)^r=EDED\cdots ED$, $r\geq0$, is left. The allowed
words of length $k$ correspond to the subsets of $\{1,2,\cdots,k\}$
having an even number of members (namely the set of positions occupied
by $E$ and $D$ in the word). Hence the number of allowed words of length
$k$ is the number of such subsets of $\{1,2,\cdots,k\}$, namely
$2^{k-1}$.

We shall show that
\zz\label{b5}J_k(t_0,z_0)=\sum_{j=1}^{2^{k-1}}\pm I_{S^{(j)}}(t_0,z_0)
\end{equation} where each $S^{(j)}$ is an allowed word of length $k$ (in fact
each allowed word of length $k$ appears exactly once in this sum, but we do
not need this fact). The proof will then be completed by obtaining a bound
for $I_S$.

We prove (\ref{b5}) by induction on $k$. So, assuming (\ref{b5}) for $J_k$,
we have 
\[\begin{split}J_{k+1}(t_0,z_0)&=\int_{t_0}^1dt_1\int g'(t_1,z_1)E(t_1-t_0,
z_1-z_0)J_k(t_1,z_1)dz_1\\
&=-\int_{t_0}^1dt_1\int g(t_1,z_1)B(t_1-t_0,z_1-z_0)J_k(t_1,z_1)dz_1\\
&\ \ \ -\int_{t_0}^1\int g(t_1,z_1)E(t_1-t_0,z_1-z_0)J_k'(t_1,z_1)dz_1
\end{split}\]
Now we observe that, if $S$ is an allowed string then $I_S'=-I_{\tilde{S}}$
where $\tilde{S}$ is defined as $BS^*$ if $S=ES^*$ and as $DS^*$ if $S=BS^*$
(note that $\tilde{S}$ is not an allowed string). Applying this to (\ref{b5})
we find $J_k'(t_0,z_0)=\sum_{j=1}^{2^{k-1}-1}\mp I_{\tilde{S}^j}(t_0,z_0)$
and then we obtain
\[J_{k+1}(t_0,z_0)=\mp\sum_{j=1}^{2^{k-1}-1}I_{BS^j}(t_0,z_0)\pm\sum_{j=1}
^{2^{k-1}-1}I_{E\tilde{S}^j}(t_0,z_0)\]
Noting that, if $S$ is an allowed string, $BS$ and $E\tilde{S}$ are also
allowed, this completes the inductive proof of (\ref{b5}).

We now proceed to the estimation of $I_S(t_0,z_0)$, when $S$ is an
allowed string. We start with some preliminary lemmas.
\begin{lemma}\label{bl1} There is a constant $C$ such that, if $\phi$ and
$h$ are real-valued Borel functions on $[0,1]\times\R$ with $|\phi(t,y)|\leq
e^{-y^2/3t}$ and $|h(t,y)|\leq1$ everywhere, then
\[\left|\int_{1/2}^1dt\int_{t/2}^tds\int_\R\int_\R\phi(s,z)h(t,y)
D(t-s,y-z)dydz\right|\leq C\]
\end{lemma}

\begin{proof} Denote the above integral by $I$. For $l\in\Z$, let $\chi_l$
be the characteristic function of the interval $[l,l+1)$ and define $\phi_l
(s,y)=\phi(s,y)\chi_l(y)$, and similarly $h_l$. Let $I_{lm}$ denote the
integral $I$ with $\phi,h$ replaced by $\phi_l,h_m$. Then we have $I=\sum_
{l,m\in\Z}I_{lm}$. Let $C_1,C_2,\cdots$ denote positive absolute constants.

Now if $|l-m|=k\geq2$ then for $z\in[l,l+1)$ and $y\in[m,m+1)$ we have
$|z-y|\geq k-1$ and then it follows easily that
\[\left|D(t-s,y-z)\right|\leq C_1e^{-(k-2)^2/4}\]
and hence $I_{lm}\leq C_2e^{-l^2/8}e^{-(k-2)^2/4}$ from which we deduce
\[\sum_{|l-m|\geq2}|I_{lm}|\leq C_3\]

Now suppose $|l-m|\leq1$. We use $\hat{\phi}_l(s,u)$ for the Fourier
transform in the second variable, and similarly $\hat{h}_m$. We note that
$\int\hat{\phi}_l(s,u)^2du=\int\phi_l(s,z)^2dz\leq C_4e^{-|l|^2/6}$ for $0
\leq s\leq1$ and similarly $\int\hat{h}_m(t,u)^2du\leq1$. We have
\[I_{lm}=\int_{1/2}^1dt\int_{t/2}^tds\int_\R\hat{\phi}_l
(s,u)\hat{h}_m(t,-u)e^{-(t-s)|u|^2/2}u^2du\]
Applying $ab\leq\frac12(a^2c+b^2c^{-1}$ with $a=\hat{\phi}_l(s,u)$, $b=\hat
{h}_m(t,-u)$ and $c=e^{l^2/12}$, we deduce that
\[\begin{split}|I_{lm}|\leq
&\int_{1/2}^1dt\int_{t/2}^tds\int_\R\hat{\phi}_l(s,u)^2
e^{l^2/12}u^2e^{-(t-s)u^2/2}du\\&+\int_{1/2}^1dt\int_{t/2}^tds\int_\R
\hat{h}_m(-t,u)^2e^{-l^2/12}u^2e^{-(t-s)u^2/2}du\end{split}\]
In the first integral we integrate first w.r.t. $t$ and obtain the bound
const.$e^{-l^2/12}$ for the integral. We get a similar
bound for the second integral (integrating w.r.t. $s$ first), and hence
\[|I_{lm}|\leq C_5e^{-l^2/12}\]
Summing over $l$ and $m$ such that $|l-m|\leq1$, we obtain
\[\sum_{|l-m|\leq1}|I_{lm}|\leq C_6\]
which completes the proof.
\end{proof}

\begin{cor}\label{bc1} There is an absolute constant $C$ such that
if $g$ and $h$ are Borel functions on $[0,1]\times\R$
bounded by 1 everywhere then
\[\left|\int_{1/2}^1dt\int_{t/2}^tds\int_{\R^2}g(s,z)E(s,z)h(t,y)
D(t-s,y-z)dydz\right|\leq C\]
and
\[\left|\int_{1/2}^1dt\int_{t/2}^tds\int_{\R^2}g(s,z)B(s,z)h(t,y)
D(t-s,y-z)dydz\right|\leq C\]
\end{cor}

\begin{proof} These follow easily from Lemma (\ref{bl1}), the second using
the easily verified fact that $|B(s,z)|\leq Cs^{-1/2}(e^{-z^2/3s})$.
\end{proof}

We note that $\int_\R E(t,z)dz=1$, and we have the bounds
\zz\label{b7}\int_\R|B(t,z)|dz\leq C_0t^{-1/2},\ \ \ \ \ \
\int_\R|D(t,z)|dz\leq C_0t^{-1}\end{equation}
where $C_0$ is an absolute constant.

\begin{lemma}\label{bl2} There is an absolute constant $C$ such that
if $g$ and $h$ are Borel functions on $[0,1]\times\R$
bounded by 1 everywhere, and $r\geq0$ then
\[\left|\int_{t_0}^1dt\int_{t_0}^tds\int_{\R^2}g(s,z)E(s-t_0,z)h(t,y)
D(t-s,y-z)(1-t)^rdydz\right|\leq C(1+r)^{-1}(1-t_0)^{r+1}\]
and
\[\left|\int_{t_0}^1dt\int_{t_0}^tds\int_{\R^2}g(s,z)B(s-t_0,z)h(t,y)
D(t-s,y-z)(1-t)^rdydz\right|\leq C(1+r)^{-1/2}(1-t_0)^{r+\frac12}\]
\end{lemma}

\begin{proof} Again, we let $C_1,\cdots$ be absolute constants. By using the
change of variables $t'=(t-t_0)/(1-t_0)$, $s'=(s-t_0)/(1-t_0)$, $y'=y(1-t_0)
^{-1/2}$, it suffices to prove these estimates when $t_0=0$. To do this, we
start by scaling the first part of Corollary \ref{bc1}, and get
\[\left|\int_{2^{-k-1}}^{2^{-k}}dt\int_{t/2}^tds\int_{\R^2}g(s,z)
E(s,z)h(t,y)D(t-s,y-z)(1-t)^rdydz\right|\leq C_1(1-2^{-k-1})^r2^{-k}\]
for $k=0,1,2\cdots$ and then by summing over $k$, we get
\[\left|\int_0^1dt\int_{t/2}^tds\int_{\R^2}g(s,z)
A(s,z)h(t,y)D(t-s,y-z)(1-t)^rdydz\right|\leq C_2(1+r)^{-1}\]
Moreover, from the bounds (\ref{b7}) we have
\[\begin{split}&\left|\int_0^1dt\int_0^{t/2}ds\int_{\R^2}g(s,z)
E(s,z)h(t,y)D(t-s,y-z)(1-t)^rdydz\right|\leq\\
&\leq C_3\int_0^1dt\int_0^{t/2}(t-s)^{-1}(1-t)^rds\leq
C_4(1+r)^{-1}\end{split}\]
and combining these bounds gives the first result. Similarly,
by scaling the second part of Corollary \ref{bc1}, we get
\[\left|\int_{2^{-k-1}}^{2^{-k}}dt\int_{t/2}^tds\int_{\R^2}g(s,z)
B(s,z)h(t,y)D(t-s,y-z)(1-t)^rdydz\right|\leq C_5(1-2^{-k-1})^r2^{-k/2}\]
for $k=0,1,2\cdots$ and then by summing over $k$, we get
\[\left|\int_0^1dt\int_{t/2}^tds\int_{\R^2}g(s,z)
B(s,z)h(t,y)D(t-s,y-z)(1-t)^rdydz\right|\leq C_6(1+r)^{-1/2}\]
Moreover, from the bounds (\ref{b7}) we have
\[\begin{split}&\left|\int_0^1dt\int_0^{t/2}ds\int_{\R^2}g(s,z)
B(s,z)h(t,y)D(t-s,y-z)(1-t)^rdydz\right|\leq\\
&\leq C_0\int_0^1dt\int_0^{t/2}(t-s)^{-1}(1-t)^rds\leq
C_7(1+r)^{-1/2}\end{split}\]
which give the second result.
\end{proof}

We can now complete the proof of Proposition \ref{mp} by obtaining the
required bound for $I_S(t_0,z_0)$. Again we use $C_1,C_2,\cdots$ for absolute
constants. We shall show that, for a suitable
choice of $M$, we have for any allowed string $S$ of length $k$
\zz\label{b8}|I_S(t_0,z_0)|\leq\frac{M^k}{\Gamma(\frac k2+1)}(1-t_0)^{k/2}
\end{equation}
We shall prove (\ref{b8}) by induction on $k$, provided $M$ is chosen large
enough. The case $k=0$ is immediate, so assume $k>0$ and that (\ref{b8})
holds for all allowed strings of length less than $k$. Then
there are three cases: (1) $S=BS'$ where $S'$ has length $k-1$; (2)
$S=EDS'$ where $S'$ has length $k-2$; (3) $S=EB^mDS'$ where $m\geq1$ and
$S'$ has length $k-m-2$. In each case $S'$ is an allowed string. We
consider the three cases separately.\vspace{.1cm}\\
{\bf Case 1.} In this case we have
\[\begin{split}|I_S(t_0,z_0)|&=\left|\int_{t_0}^1dt_1\int_\R
B(t_1-t_0,z_1-z_0)g(t_1,z_1)I_{S'}(t_1,z_1)dz_1\right|\\
&\leq\frac{M^{k-1}}{\Gamma(\frac{k+1}2)}\int_{t_0}^1
(1-t_1)^{(k-1)/2}dt_1\int_\R|B(t_1-t_0,z_1-z_0)|dz_1\\
&\leq\frac{C_1M^{k-1}}{\Gamma(\frac{k+1}2)}\int_{t_0}^1
(1-t_1)^{(k-1)/2}(t_1-t_0)^{-1/2}dt_1\\&=C_1\sqrt{\pi}M^{k-1}
(1-t_0)^{k/2}/\Gamma\left(\frac k2+1\right)\end{split}\]
where we have used the inductive hypothesis to bound $I_{S'}$, and then
the bound (\ref{b7}). (\ref{b8}) then follows if $M$ is large enough.\vspace{.1cm}\\
{\bf Case 2.} Now we have
\[I_S(t_0,z_0)=\int_{t_0}^1dt_1\int_{t_1}^1dt_2\int_{\R^2}g(t_1,z_1)
g(t_2,z_2)E(t_1-t_0,z_1-z_0)D(t_2-t_1,z_2-z_1)I_{S'}(t_2,z_2)dz_1dz_2\] 
We set $h(t,z)=g(t,z)I_{S'}(t,z) (1-t)^{1-\frac k2}$ so that
$\|h\|_\infty\leq M^{k-2}/\Gamma(k/2)$ by the inductive
hypothesis, and then from the first part of Lemma \ref{bl2} we deduce that
\[|I_S(t_0,z_0)|\leq\frac{C_2M^{k-2}(1-t_0)^{k/2}}{k\Gamma(k/2)}\] and
(\ref{b8}) follows if $M$ is large enough.\vspace{.1cm}\\
{\bf Case 3.} In this case have
\[\begin{split}&I_S(t_0,z_0)=\int_{t_0<t_1<\cdots<t_{m+2}<1}dt_1\cdots
dt_{m+2}\int_{\R^{m+2}}\left(\prod_{j=1}^{m+2}g(t_j,z_j)\right)
E(t_1-t_0,z_1-z_0)\times\\&\times\prod_{j=2}^{m+1}
B(t_j-t_{j-1},z_j-z_{j-1})D(t_{m+2}-t_{m+1},z_{m+2}-z_{m+1})
I_{S'}(t_{m+2},z_{m+2})dz_1\cdots dz_{m+2}\end{split}\]
Now let $h(t,z)=g(t,z)I_{S'}(t,z)(1-t_{m+2})^{(2+m-k)/2}$, so that by the
inductive hypothesis on $S'$ we have $\|h\|_\infty\leq M^{k-m-2}
/\Gamma(\frac{k-m}2)$. Then, writing
\[\begin{split}\Omega(t,z)=&\int_t^1dt_{m+1}\int_{t_{m+1}}^1dt_{M+2}\int_
{\R^2}g(t_{m+1},z_{m+1})h(t_{m+2},z_{m+2})(1-t_{m+2})^{(k-m-2)/2}\times\\&
\times B(t_{m+1}-t,z_{m+1}-z)D(t_{m+2}-t_{m+1},z_{m+2}-z_{m+1})dz_{m+1}
dz_{m+2}\end{split}\]
we find from Lemma 2 that
\[|\Omega(t,z)|\leq C_3(k-m)^{-1/2}M^{k-m-2}(1-t)^{(k-m-1)/2}
/\Gamma\left(\frac{k-m}2\right)\]
Using this in
\[\begin{split}I_S(t_0,z_0)=&\int_{t_0<t_1<\cdots<t_m<1}dt_1\cdots
dt_m\int_{\R^m}\left(\prod_{j=1}^mg(t_j,z_j)\right)
E(t_1-t_0,z_1-z_0)\times\\&\times\prod_{j=2}^mB(t_j-t_{j-1},z_j-z_{j-1})
\Omega(t_m,z_m)dz_1\cdots dz_m\end{split}\] 
and using the bounds (\ref{b7}) we find
\[\begin{split}|I_S(t_0,z_0)|&\leq C_4^{m+1}(k-m)^{-1/2}\frac{M^{k-m-2}}
{\Gamma(\frac{k-m}2)}\int_{t_0<t_1<\cdots<t_m<1}(t_2-t_1)^{-1/2}
\cdots\\&\cdots(t_m-t_{m-1})^{-1/2}(1-t_m)^{(k-m-1)/2}dt_1\cdots dt_m\\
&=C_4^{m+1}(k-m)^{-1/2}\frac{M^{k-m-2}\pi^{(m-1)/2}\Gamma(\frac{k-m+1}2)}
{\Gamma(\frac{k-m}2)\Gamma(\frac k2+1)}(1-t_0)^{k/2}\end{split}\]
from which again (\ref{b8}) follows, provided $M$ is large enough. Putting
(\ref{b8}) with $t_0=0$, $z_0=0$ and $k=p$ in (\ref{b5})
completes the proof of Proposition \ref{mp}.
\end{proof}

\noindent
{\em Proof of Proposition \ref{be}}. We first
note that it suffices to prove it for $d=1$. To see this let $g,W,x$ be
as in the statement of Proposition \ref{be}. By a rotation of coordinates
we can suppose $x=(\alpha,0,\cdots,0)$. Then for fixed Brownian paths $W_2,
\cdots,W_d$ we can define $h$ on $[0,1]\times\R$ by $h(t,u)=g(t,u,W_2(t),
\cdots,W_d(t))$ and the $d=1$ case of the Proposition gives
\[\IE\left(\int_0^1\{h(t,W_1(t)+\alpha)-h(t,W_1(t))\}dt\right)^p\leq C^p(p/2)!
|\alpha|^p\]
and then the required result follows by averaging over $W_2,\cdots,W_d$.

So we suppose $d=1$. Given a Borel function $g$ on $[0,1]\times\R$ with
$|g|\leq1$ we can find a sequence of compactly supprted smooth functions $g_n$ with $|g_n|
\leq1$, converging to $g$ a.e. on $[0,1]\times\R$. Then $g_n(t,W(t))\rightarrow
g(t,W(t))$ a.s. for a.a. $t\in[0,1]$, and the same for $g_n(t,W(t)+x)$, so by
Fatou's lemma it suffices to prove the proposition for smooth $g$. But then
we have $g(t,W(t)+x)-g(t,W(t))=\int_0^xg'(t,W(t)+u)du$ and we can apply
Proposition \ref{mp} and Minkowsi's inequality to conclude the proof of
Proposition \ref{be}.

What we in fact need is a scaled version of Proposition \ref{be} for subintervals
of [0,1]. For $s\geq0$ we denote by ${\mathcal F}_s$ the $\sigma$-field generated by $\{W(\tau):0<\tau<s\}$. Then we can
state the required result:

\begin{cor}\label{bc2} Let $g$ be a Borel function on $[0,1]\times\R^d$
with $|g|\leq1$ everywhere. Let $0\leq s\leq a<b\leq1$ and let $\rho(x)=
\int_a^b\{g(W(t)+x)-g(W(t))\}dt$. Then for $x\in\R^d$ and $\lambda>0$ we have
\[\IP(|\rho(x)|\geq\lambda l^{1/2}|x|\ |{\mathcal F}_s)\leq2e^{-\lambda^2
/(2C^2)}\]
where $l=b-a$ and $C$ is the constant in Proposition \ref{be}.
\end{cor}
\begin{proof} First assume $s=a=0$, $b=1$. Let $\alpha=(2C^2|x|^2)^{-1}$. Then
\[\IE(e^{\alpha\rho(x)^2})=\sum_{k=0}^\infty\frac{\alpha^k}{k!}
\IE(\rho(x)^{2k})\leq\sum_{k=0}^\infty\alpha^kC^{2k}|x|^{2k}=2\]
and so $\IP(|\rho(x)|\geq\lambda|x|)=\IP(e^{\alpha\rho(x)^2}\geq
e^{\alpha\lambda^2|x|^2})\leq2e^{-\alpha\lambda^2|x|^2}
=2e^{-\lambda^2/(2C^2)}$.

For the general case, let $\tilde{W}(t)=l^{-1/2}\{W(a+tl)-W(a)\}$, so that
$\tilde{W}$ is a standard Brownian motion, and let $h(x)=g(W(a)+x)$. Then
$\rho(x)=l^{1/2}\int_0^1\{h(\tilde{W}(t)+x)-h(\tilde{W}(t))\}dt$ and it
follows from the first part that $\IP(|\rho(x)|\geq\lambda l^{1/2}|x|\ |
{\mathcal F}_a)\leq2e^{-\lambda^2/(2C^2)}$. The required result then follows
by taking conditional expectation w.r.t. ${\mathcal F}_s$.
\end{proof}

We note that the unconditional bound
\[\IP(|\rho(x)|\geq\lambda l^{1/2}|x|)\leq2e^{-\lambda^2/(2C^2)}\]
follows by taking $s=0$. Also in the same way we obtain, for any even $p\in\mathbb{N}$,
\zz\label{bp}\IE(\rho(x)^p|{\mathcal F}_s)\leq C^pl^{p/2}(p/2)!|x|^p\z

The following lemma will also be needed:
\begin{lemma}\label{bl3} If $p>1+\frac d2$ there is a constant $c(p,d)$ such
that is $g\in L^p([0,1]\times\R^d)$ then
\[\IE\left(\int_0^1g(t,W(t))dt\right)^2\leq c(p,d)\|g\|_p^2\]
\end{lemma}

\begin{proof}We have
\[\IE\left(\int_0^1g(t,W(t))dt\right)^2=2\int_0^1dt\int_0^tds
\int_{\R^{2d}}g(s,\zeta)g(t,z)E(s,\zeta)E(t-s,z-\zeta)d\zeta dz\]
Now, if $q=\frac p{p-1}$ then $\int E(t,z)^qdz=O(t^{-(q-1)d/2})$ and
$p>1+\frac d2$ implies $(q-1)d/2<1$, so the result follows from
H\"{o}lder's inequality.
\end{proof}

\section{Proof of Theorem}\label{pf}

We now apply Corollary \ref{bc2} and Lemma \ref{bl3} to the proof of the
theorem. First we give a brief sketch of the proof.\vspace{.2cm}\\
{\bf Outline of proof.} The proof is motivated by the elementary case
when $f$ is Lipschitz in the second variable. In this case, if $I=[a,b]$
is a subinterval of [0,1] and $u$ is a solution of (\ref{eq4}) satisfying
\zz\label{eq5}|u(t)|\leq\alpha,\ \ \ \ \ t\in I\end{equation}
and $\beta=|u(a)|$, then we deduce from (\ref{eq5}) that
$|u(t)|\leq\alpha'=\beta+L|I|\alpha$ for $t\in I$, where $L$ is the
Lipschitz constant, i.e. (\ref{eq5}) holds with $\alpha$ replaced by $\alpha'$.
If $L|I|<1$ it follows that (\ref{eq5}) holds with $\alpha=(1-L|I|)^{-1}\beta$,
and of course if $\beta=0$ this gives $u=0$ on $I$.

We try to copy this argument using Corollary \ref{bc2} as a substitute for a
Lipschitz condition. There are two difficulties: first, Corollary \ref{bc2}
is a statement about probabilities and we need an `almost sure' version, and
in doing so we lose something; second, in Corollary \ref{bc2}, $x$ is a
constant, whereas we are dealing with a function $u$ depending on $t$.
The way round the second problem is to approximate $u$ by a sequence of
step functions $u_l$ and then use
\zz\label{aa}\begin{split}&\int_I\{f(W(t)+u(t))-f(W(t))\}dt=\lim_{l\rightarrow\infty}
\int_I\{f(W(t)+u_l(t))-f(W(t))\}dt\\
&=\int_I\{f(W(t)+u_n(t))-f(W(t))\}dt+\sum_{l=n}^\infty
\int_I\{f(W(t)+u_{l+1}(t))-f(W(t)+u_l(t))\}dt\end{split}\z
where $u_n$ is constant on the interval $I$, and then to apply the `almost
sure' form of the
proposition to each interval of constancy of the terms on the right. 
Again, we lose something in
doing this, but, as it turns out, we still have good enough estimates to
prove the theorem. In fact, we need two versions of the `almost sure'
(nearly) Lipschitz condition, the first to estimate
$\int\{f(W(t)+u_n(t))-f(W(t))\}dt$ and the second to estimate
$\int\{f(W(t)+u_{l+1}(t))-f(W(t)+u_l(t))\}dt$. We also need a third estimate, for sums of integrals of the second type.

The two versions of the `almost sure' nearly-Lipschitz condition are
conditions (\ref{eq7}) and (\ref{eq9}) below, and the third estimate is (\ref{eqx}). In Lemmas
\ref{l1}, \ref{l3}, \ref{nla} and \ref{nlb} it is shown that these conditions indeed hold almost surely.
Lemmas \ref{l4} and \ref{l5} establish a technical condition (\ref{eq8})
needed to justify the passage to the limit as $l\rightarrow
\infty$ (which is not trivial when $f$ is not continuous). With these
preliminaries the above programme is carried out in Lemma \ref{ll}. The
analogue of (\ref{eq5}) above is (\ref{eq10}). We no longer immediately get
$\alpha=0$ when $\beta=0$, but we get a good enough bound to prove the
uniqueness of the solution to (\ref{eq1}), for any $W$ satisfying
(\ref{eq7},\ref{eq9},\ref{eq8},\ref{eqx}).\vspace{.2cm}

We now turn to the details.\vspace{.2cm}

For any $n\geq0$ we can divide [0,1] into $2^n$ intervals
$I_{nk}=[k2^{-n},(k+1)2^{-n}]$, $k=0,1,2,\cdots,2^n-1$. We shall also
consider dyadic decompositions of $\R^d$, and say $x\in\R^d$ is a {\em
dyadic} point if each component of $x$ is rational with denominator a
power of 2. Let $Q=\{x\in\R^d:\|x\|\leq1\}$, where $\|x\|$ denotes the
supremum norm $\max_{1\leq j\leq d}|x_j|$. We also introduce the
notation
\[\sigma_{nk}(x)=\int_{I_{nk}}\{g(W(t)+x)-g(W(t))\}dt\]and
\[\rho_{nk}(x,y)=\sigma_{nk}(x)-\sigma_{nk}(y)=\int_{I_{nk}}\{g(W(t)+x)-
g(W(t)+y)\}dt\]Then we can state:

\begin{lemma}\label{l1}Let $g$ be a real function on $[0,1]\times\R^d$ with $|g(t,z)|\leq1$ everywhere. Then with probability 1 we can
find $C>0$ so that
\zz\label{eq7}|\rho_{nk}(x,y)|\leq C\left\{n^{1/2}+\left(\log^+\frac1{|x-y|}\right)^{1/2}\right\}2^{-n/2}|x-y|\z
for all dyadic $x,y\in Q$ and all choices of integers $n,k$ with $n>0$ and $0\leq k\leq2^n-1$.
\end{lemma}

\begin{proof} Let us say that two dyadic points $x,y\in\R^d$ are {\em
dyadic neighbours} if for some integer $m\geq0$ we have $\|x-y\|=2^{-m}$
and $2^{-m}x,2^{-m}y\in\Z^d$. Then using the Corollary \ref{bc2} we have,
for any such pair $x,y\in Q$ and any $n,k$ that
\[\IP\left(|\rho_{nk}(x,y)|\geq\lambda(n^{1/2}+m^{1/2})2^{-m-n/2}\right)
\leq C_1e^{-C_2\lambda^2(n+m)}\]
and by summing over all possible choices of $n,k,m,x,y$ we find that the
probability that
\[|\rho_{nk}(x,y)|\geq\lambda(n^{1/2}+m^{1/2})2^{-m-n/2}\]
for some choice of $I_{nk}$ and dyadic neighbours $x,y\in Q$ is not more
than\\$\sum_{n=1}^\infty\sum_{m=0}^\infty 2^n3^d2^{d(m+3)}C_1e^{-C_2\lambda^2(1+m+n)}$
which approaches 0 as $\lambda\rightarrow\infty$.

It follows that, given $\E>0$, we can find $\lambda(\E)$ such that, with
probability $>1-\E$, we have
\[|\rho_{nk}(x,y)|<\lambda(1+n^{1/2}+m^{1/2})2^{-m-n/2}\]
for all choices of $n,k$ and dyadic neighbours in $Q$.

Next let $x,y$ be any two dyadic points in $Q$, and let $m$ be the
smallest non-negative integer such that $\|x-y\|<2^{-m}$. For $r\geq m$,
choose $x_r$ to minimise $\|x-x_r\|$ subject to $2^rx_r\in\Z^d$, and
$y_r$ similarly. Then $\|x_m-y_m\|=2^{-m}$ or 0, and for $r\geq m$,
$\|x_r-x_{r+1}\|=2^{-r-1}$ or 0. So $x_m,y_m$ are dyadic neighbours or
equal, and the same applies to $x_r,x_{r+1}$ and $y_r,y_{r+1}$. Then we
have
\[\rho_{nk}(x,y)=\rho_{nk}(x_m,y_m)+\sum_{r=m}^\infty\rho_{nk}(x_r,x_m)
+\sum_{r=m}^\infty\rho_{nk}(y_m,y_r)\]
(note that the sums are actually finite, since $x,y$ are dyadic, so that
$x=x_r$ and $y=y_r$ for large $r$). Then
applying the above bounds for the case of dyadic neighbours to each
term, we get the desired result.
\end{proof}

Next we prove a similar estimate for $\sigma_{nk}$, which is analogous to the Law of the Iterated Logarithm for Brownian motion.

\begin{lemma}\label{l3} With probability 1 there is a constant $C>0$ such that for all $n\in\mathbb{N}$, $k\in\{0,1,\cdots,2^n-1\}$
and dyadic $x\in Q$ we have
\zz\label{eq9}|\sigma_{nk}(x)|\leq Cn^{1/2}2^{-n/2}(|x|+2^{-2^n})\z
\end{lemma}

\begin{proof} For any integer $r\geq0$ we let $Q_r=\{x\in\R^d: \|x\|\leq2^{-r}\}$. Then if $m\geq r$ the
number of pairs $(x,y)$ of dyadic neighbours in $Q_r$ with
$\|x-y\|=2^{-m}$ is $\leq(9\times2^{m-r})^d$ and for each such pair we
have 
\[\IP(|\rho_{nk}(x,y)|\geq\lambda(n^{1/2}+\sqrt{m-r})2^{-m-n/2})\leq
C_1e^{-C_2\lambda^2(n+m-r)}\leq C_12^{2d(r-m)}e^{-C_2\lambda^2}e^{-n}\]
for $\lambda$ large. By summing over $n$, $1\leq r\leq2^n$ and $m\geq r$ and all pairs $(x,y)$, we deduce that, with
probability $\geq 1-C_3e^{-C_4\lambda^2}$, we have
$\rho_{nk}(x,y)\leq\lambda(n^{1/2}+\sqrt{m-r})2^{-r-n/2}$ for $n\in\mathbb{N}$, $1\leq r\leq n$ and $m\geq r$ and all pairs
$(x,y)$ of dyadic neighbours in $Q_r$ with $\|x-y\|=2^{-m}$, and then, by an argument similar
to Lemma \ref{l1}, we get for all $n$ and $1\leq r\leq n$ that $\sigma_{nk}(x)\leq C_5\lambda n^{1/2}2^{-r-n/2}$ for all dyadic
$x\in Q_r$. The required result follows.
\end{proof}

The next two lemmas are used to justify the passage to the limit $l\rightarrow\infty$ in (\ref{aa}).

Let $\Phi$ denote the set of $Q$-valued functions $u$ on [0,1] satisfying
$|u(s)-u(t)|\leq|s-t|$, $s,t\in[0,1]$, and let $\Phi_n$ denote the set of
$Q$-valued functions on [0,1] which are constant on each $I_{nk}$ and
satisfy $|u(k2^{-n})-u(l2^{-n})|\leq|k-l|2^{-n}$. Then let $\Phi^*=\Phi
\cup\cup_n\Phi_n$.

\begin{lemma}\label{l4} Given $\E>0$, we can find $\eta>0$ such that if
$U\subset(0,1)\times\R^d$ is open with $|U|<\eta$, then, with
probability $\geq1-\E$, we have $\int_0^1\chi_U(t,W(t)+u(t))dt
\leq\E$ for all $u\in\Phi^*$.
\end{lemma}

\begin{proof} Fix $\E>0$. By Lemma \ref{l1} we can find $K$ such that, for
any Borel function $\phi$ on $[0,1]\times\R^d$ with $|\phi|\leq1$ everywhere
we have with probability $>1-\E/2$ that
\zz\label{eq55}\int_{I_{kn}}\{\phi(W(t)+x)-\phi(W(t)+y)\}dt\leq Kn^{1/2}
2^{-3n/2}\z for all pairs of dyadic points $x,y$ in $Q$ and all choices of
$n,k$. Then we choose $m$ such that $4K\sum_{n=m}^\infty n^{1/2}2^{-n/2}<\E$.
Let $\Omega$ be a finite set of dyadic points of $Q$ such that every $x\in Q$
is within distance $2^{-m}$ of some point of $\Omega$.

Provided $\delta$ is chosen small enough, any bounded Borel function
$\phi$ on $[0,1]\times\R^d$ with
$\|\phi\|_{L^p([0,1]\times\R^d)}<\delta$ will satisfy \[\IP\left(\left|
\int_{I_{mk}}\phi(t,W(t)+x)dt\right|\geq2^{-m}\E/4\right)<\frac\E{2^{m+1}
\#(\Omega)}\]
for each $k,x$. Then the probability that
\zz\label{eq6}\left|\int_{I_{mk}}\phi(t,W(t)+x)dt\right|<2^{-m}\E/4\ \ {\rm
for}\ {\rm every}\ \ k\in\{0,1,\cdots,2^m-1\},\ \ x\in\Omega\end{equation}
is at least $1-\E/2$.

Now let $\eta=\delta^p$, and suppose $U$ is open with $m(U)<\eta$. Let
$(\phi_r)$ be an increasing sequence of continuous non-negative functions
on $[0,1]\times\R^d$, converging pointwise to $\chi_U$. Note that
then $\|\phi_r\|_{L^p([0,1]\times\R^d}<\delta$. For each $r$ define
events $A_r$: (\ref{eq6}) holds for $\phi=\phi_r$ and $B_r$: (\ref{eq55})
holds for $\phi=\phi_r$. Then $\IP(A_r)\geq1-\E/2$ and
$\IP(B_r)\geq1-\E/2$. Also, when $A_r$ and $B_r$ both hold, we have
$\int_{I_{km}}\phi_r(t,W(t)+x)dt<2^{-m}\E/2$ for all $x$ such that
$|x|\leq2$.

Now let $u\in\Phi^*$. For each $n\geq m$ choose $u_n\in\Phi_n$ taking
a constant dyadic value within $2^{-n}$ of $u(k2^{-n})$ on
$I_{nk}$ for $k=0,1,\cdots,2^n-1$. Now if $A_r$ and $B_r$ hold then
$\int_0^1\phi_r(t,W(t)+u_m(t))dt\leq\E/2$ and
\[\left|\int_0^1\{\phi_r(t,W(t)+u_n(t))-\phi_r(t,W(t)+u_{n+1}(t))\}dt
\right|\leq Kn^{1/2}2^{-n/2}\] from which it follows that
$\int_0^1\phi_r(t,W(t)+u(t))dt<\E$. So if we define the event $Q_r:$
$\int_0^1\phi_r(t,W(t)+u(t))dt\leq\E$ for all $u\in\phi$, then we have
$\IP(Q_r)\geq1-\E$.
But since $\phi_{r+1}\geq\phi_r$ we have $Q_{r+1}\subseteq Q_r$, and it
follows that with probability $\geq1-\E$ we have $Q_r$ for all $r$, from
which the result follows, since
$\int_0^1\phi_r(t,W(t)+u(t))dt\rightarrow\int_0^1\chi_U(t,W(t)+u(t))dt$
by the bounded convergence theorem.
\end{proof}

\begin{lemma}\label{l5} If $g$ is a bounded Borel function on
$[0,1]\times\R^d$, then, with probability 1, whenever $(u_n)$ is a
sequence in $\Phi^*$ converging pointwise to a limit $u\in\Phi^*$, we have
\zz\label{eq8}\int_0^1g(t,W(t)+u_n(t))dt\rightarrow\int_0^1g(t,W(t)+u(t))dt\z
\end{lemma}

\begin{proof} Given $\E>0$, let $\eta$ be as in Lemma \ref{l4}, and let $h$ be a
bounded continuous function on $[0,1]\times\R^d$ such that $g=h$ outside
an open set $U$ with $m(U)<\eta$. With probability $\geq1-\E$, the
conclusion of Lemma \ref{l4} holds, which means that for any convergent
sequence $(u_n)$ in $\Phi$ we have $\int_0^1\mathbb{I}_U(t,W(t)+u_n(t))dt
\leq\E$, and the same for the limit $u(t)$, so, if $M$ is an upper bound for
$|g-h|$, we have the bound $\left|\int\{g(t,W(t)+u_n(t))-h(t,W(t)+u_n(t))\}dt
\right|\leq M\E$, and the same for $x$ in place of $u_n$. Also, since
$h$ is continuous, $\int_0^1h(t,W(t)+u_n(t))dt\rightarrow
\int_0^1h(t,W(t)+u(t)dt$. It follows that, for $n$ large enough,
$\left|\int_0^1g(t,W(t)+u_n(t))dt-\int_0^1g(t,W(t)+u(t)dt\right|<(2M+1)\E$,
and, since this holds for any $\E>0$, the result follows.
\end{proof}

Note that Lemma \ref{l5} implies that $\rho_{nk}(x,y)$ and $\rho_{nk}(x)$ are continuous, so that 
the estimates of Lemmas \ref{l1} and \ref{l3} will hold for all $x,y\in Q$.

We also need a stronger bound for sums of $\rho_{nk}$ terms than that given by the bounds for individual terms in Lemma \ref{l1}, and
the next two lemmas provide this. They are motivated by the idea that any solution of (\ref{eq4}) should satisfy the approximate
equation $u((k+1)2^{-n})\approx u(k2^{-n})+\sigma_{nk}(u(k2^{-n}))$ which suggests that on a short time interval a solution can
be approximated by an `Euler scheme' $x_{k+1}=x_k+\sigma_{nk}(x_k)$.

\begin{lemma}\label{nla} Given even $p\geq2$ we can find $C>0$ such that, for any choice of $n,r\in\mathbb{N}$ with $r\leq2^{n/2}$,
$k\in\{0,1,\cdots,2^n-r\}$ and $x_0\in Q$, if we define $x_1,\cdots,x_r$ by the recurrence relation $x_{q+1}=x_q+\sigma_{n,k+q}(x_q)$, then
\[\IP\left(\sum_{q=1}^r|\rho_{n,k+q}(x_{q-1},x_q)|\geq2^{-n}\left\{C\sum_{q=0}^{r-1}|x_q|+\lambda r^{1/2}|x_0|\right\}\right)\leq
C\lambda^{-p}\]
for any $\lambda>0$.
\end{lemma}
\begin{proof} We use $C_1,\cdots$ to denote constants which depend only on $d$ and $p$. We write ${\mathcal F}_j$ for ${\mathcal F}
_{(k+j)2^{-n}}$. Note first that $x_q$ is ${\mathcal F}_q$ measurable and $\IE(|\sigma_{n,k+q}(x_q)|^p|{\mathcal F}_q)\leq
C_12^{-np/2}|x_q|^p$ by (\ref{bp}). Hence $\IE|\sigma_{n,k+q}(x_q)|^p\leq C_12^{-np/2}\IE|x_q|^p$. It follows that
$\IE|x_{q+1}|^p\leq(1+C^{1/p}_12^{-n/2})^p\IE|x_q|^p$ and so
\zz\label{nq1}\IE|x_q|^p\leq(1+C^{1/p}_12^{-n/2})^p|x_0|^p\leq C_2|x_0|^p\z
for $1\leq q\leq r$.

Now let $Y_q=|\rho_{n,k+q}(x_{q-1},x_q)|$, $Z_q=\IE(Y_q|{\mathcal F}_q)$ and $X_q=Y_q-Z_q$. Then $X_q$ is ${\mathcal F}_{q+1}$ measurable
and $\IE(X_q|{\mathcal F}_q)=0$ so by Burkholder's inequality
\[\begin{split}\IE|\sum_{q=1}^rX_q|^p&\leq C_3\IE(\sum X_q^2)^{p/2}\leq C_3r^{p/2-1}\IE\sum|X_q|^p\leq C_4r^{p/2-1}\sum\IE(Y_q^p)\\
&\leq C_5r^{p/2-1}2^{-np/2}\sum\IE|x_q-x_{q-1}|^p=C_5r^{p/2-1}2^{-np/2}\sum\IE|\sigma_{n,k+q-1}(x_{q-1})|^p\\
&\leq C_6r^{p/2-1}2^{-np}\sum_{q=1}^r\IE|x_{q-1}|^p\end{split}\]
from which we deduce using (\ref{nq1}) that
\zz\label{nq2}\IE|\sum_{q=1}^rX_q|^p\leq C_7r^{p/2}2^{-np}|x_0|^p\z
Also let $V_q=\IE(Z_q|{\mathcal F}_{q-1})$ and $W_q=Z_q-V_q$. Noting that $Z_q\leq C_82^{-n/2}\sigma_{n,q-1}(x_{q-1})$ we get in a
similar way that
\zz\label{nq3}\IE|\sum W_q|^p\leq C_9r^{p/2}2^{-np}|x_0|^p\z
We also have 
\zz\label{nq4}|V_q|\leq C_{10}2^{-N}|X_{q-1}|\z
Now $Y_q=X_q+W_q+V_q$. By (\ref{nq2}) and (\ref{nq3}) we have $\IP(|\sum_{q=1}^r(X_q+W_q)|>2^{-n}\lambda r^{1/2}|x_0|)\leq C_{11}\lambda^p$
and the result then follows by (\ref{nq4}).
\end{proof}

\begin{lemma}\label{nlb} With probability 1 there exists $C>0$ such that for any $n,r\in{\mathbb N}$ with $r\leq2^{n/4}$,
any $k\in\{0,1,\cdots,2^n-r\}$ and any $y_0,\cdots,y_r\in Q$ we have
\zz\label{eqx}\sum_{q=1}^r|\rho_{n,k+q}(y_{q-1},y_q)|\leq C\left(2^{-3n/4}|y_0|+2^{-n/4}\sum_{q=0}^{r-1}|\gamma_q|+2^{-2^{n/2}}\right)\z
where $\gamma_q=y_{q+1}-y_q-\sigma_{n,k+q}(y_q)$.
\end{lemma}
\begin{proof} Let $\delta_n=2^{-2^{n/2}}$. By Lemma \ref{l1}, with probability 1 there exists $C>0$ such that, for any $n,k\geq0$ and any
$x,y\in Q$, we have
\zz\label{nq5}\rho_{nk}(x,y)\leq C2^{-n/4}|x-y|+\delta_n\z
As before, let $Q_s=\{x\in\R^d:\|x\|\leq2^{-s}\}$. Then, for integers $s$ with $0\leq s<2^{n/2}$, let $\Omega_{ns}$ be a set of not more
than $(2^nd^{1/2})^d$ points of $Q_s$ such that every $x\in Q_s$ is within distance $2^{-s-n}$ of a point of $\Omega_{ns}$ and let
$\Omega_n=\cup_{0\leq s<2^{-n/2}}\Omega_{ns}$. Let $p=8(4+d)$. Then by Lemma \ref{nla} there is $C_1>0$ such that the probability that
\[\sum_{q=1}^r|\rho_{n,k+q}(x_{q-1},x_q)|\geq2^{-n}\left(C_1\sum_{q=0}^{r-1}|x_q|+\lambda2^{n/8}r^{1/2}|x_0|\right)\]
for some $n,r,k$ as in the statement and some $x_0\in\Omega_n$, is bounded above by $C_1\sum_{n=0}^\infty\lambda^{-p}2^{n(3+d)}2^{-pn/8}$
which approaches 0 as $\lambda\rightarrow\infty$. Hence with probability 1 there exists $C>0$ such that
\zz\label{nq6}\sum_{q=1}^r|\rho_{n,k+q}(x_{q-1},x_q)|<C2^{-n}\left(\sum_{q=0}^{r-1}|x_q|+2^{n/8}r^{1/2}|x_0|\right)\z
for all $n,k,r$ as above and $x_0\in\Omega_n$.

We now suppose, as we may with probability 1, that (\ref{nq5}) and (\ref{nq6}) hold (with the same $C$). We fix $n,k,r,y_0\cdots
y_r,\gamma_0\cdots\gamma_r$ as in the statement of the lemma. Take the smallest $s$ such that $y_0\in Q_s$, noting that then
$2^{-s-1}\leq|y_0|\leq d^{1/2}2^{-s}$. Then we find $x_0\in\Omega_{ns}$ with $|x_0-y_0|<2^{-s-n}\leq2^{1-n}|y_0|$ and define $x_1\cdots
x_r$ by the recurrence relation $x_{q+1}=x_q+\sigma_{n,k+q}(x_q)$. Then by (\ref{nq6})
\[\sum_{q=1}^r|\rho_{n,k+q}(x_{q-1},x_q)|<C2^{-n}\left(\sum_{q=0}^{r-1}|x_q|+2^{n/4}|x_0|\right)\]
Using (\ref{nq5}) we have $|x_{q+1}|=|x_q+\sigma_{n,k+q}(x_q)|\leq(1+C2^{-n/4})|x_q|+\delta_n$ so $|x_q|\leq C_1(|x_0|+r\delta_n)$ and
\zz\label{nq7}\sum_{q=1}^r|\rho_{n,k+q}(x_{q-1},x_q)|<C_22^{-3n/4}(|x_0|+2^{n/4}\delta_n)\z
Now let $u_q=x_q-y_q$. Then $|u_{q+1}-u_q|\leq|\rho_{n,k+q}(x_q,y_q)|+|\gamma_q|$ so
\[|u_{q+1}|\leq|u_q|(1+C2^{-n/4})+|\gamma_q|+\delta_n\]
and since $|u_0|\leq2^{1-n}|y_0|$ we deduce that $|u_q|\leq C_3(2^{-n}|y_0|+r\delta_n+\sum_{q=0}^{r-1}|\gamma_q|)$ and so
\zz\label{nq8}|\rho_{n,k+q}(x_q,y_q)|\leq C_42^{-n/4}\left(2^{-n}|y_0|+r\delta_n+\sum_{q=0}^{r-1}|\gamma_q|\right)\z
and we have the same bound for $|\rho_{n,k+q}(x_{q-1},y_{q-1})|$. Now
\[\rho_{n,k+q}(y_{q-1},y_q)=\rho_{n,k+q}(x_{q-1},x_q)+\rho_{n,k+q}(y_{q-1},x_{q-1})+\rho_{n,k+q}(x_q,y_q)\]
and then using (\ref{nq7}), (\ref{nq8}) and the fact that $|x_0-y_0|\leq2^{1-n}|y_0|$ we deduce that
\[\sum_{q=1}^r|\rho_{n,k+q}(y_{q-1},y_q|\leq C_5\left(2^{-3n/4}|y_0|+2^{-n/4}\sum_{q=0}^{r-1}|\gamma_j|+2^{-n/2}\delta_n\right)\]
from which the result follows.
\end{proof}

We now proceed to complete the proof of the theorem. From now on we take
$g=f$ in the definition of $\sigma_{nk}$ and $\rho_{nk}$. We consider a Brownian path $W$ satisfying the conclusions of Lemmas \ref{l1},
\ref{l3}, \ref{nlb} and \ref{l5} for some $C>0$. We shall show that for such a Brownian path the only solution $u$ of (\ref{eq4})
in $\Phi$ is $u=0$. This will follow from the following:

\begin{lemma}\label{ll} Suppose $W$ satisfies the conclusions of Lemmas \ref{l1},
\ref{l3}, \ref{nlb} and \ref{l5} for some $C>0$.
Then there are positive constants $K$ and $m_0$ such that, for all
integers $m>m_0$, if $u$ is a solution of (\ref{eq4}) in $\Phi$ and for some $j\in\{0,1,\cdots,2^m-1$ and some $\beta$ with
$2^{-2^{3m/4}}\leq\beta\leq2^{-2^{2m/3}}$ we have $|u(j2^{-m})|\leq\beta$, then
\[|u((j+1)2^{-m})|\leq\beta\{1+K2^{-m}\log(1/\beta)\}\]
\end{lemma}

\begin{proof} We use $C_1,C_2,\cdots$ for positive constants which depend
only on the constant $C$ and the dimension $d$. Fix $m$, $j$ and $\beta$ as in the statement, and suppose $|u(j2^{-m})|\leq\beta$. Let
$N$ be the integer part of $4\log_2(1/\beta)$. Suppose $u\in\Phi$ satisfies (\ref
{eq4}), and let $u_n$ be the step function which takes the constant value
$u(k2^{-n})$ on the interval $I_{nk}$, for $k=0,1,\cdots,2^n-1$.

Let $\alpha$ be the smallest nonnegative number such that
\zz\label{eq10}\sum_{k=j2^{n-m}}^{(j+1)2^{n-m}-1}|u((k+1)2^{-n})-u(k2^{-n})|
\leq\alpha2^{-m}(n^{1/2}2^{n/2}+N)\end{equation}
for all $n$ with $m\leq n\leq N$.

For $n\geq m$ let
\zz\label{ps}\psi_n=\sum_{k=j2^{n-m}}^{(j+1)2^{n-m}-1}|u(k2^{-n})|\end{equation}
Then by (\ref{eq10})
\[\psi_n\leq2\psi_{n-1}+\alpha2^{-m}(n^{1/2}2^{n/2}+N)\]
for $n>m$,and since $\psi_m=\beta$ it follows that
\zz\label{eq11}\psi_n\leq2^{n-m}\beta+\sum_{l=m+1}^n\alpha2^{n-l-m}(l^{1/2}2^{l/2}+N)\leq C_12^{n-m}(\beta+\alpha2^{-m}N)\end{equation}
for all $n$ with $m\leq n\leq N$, where we have used the fact that $m^{1/2}2^{m/2}$ is bounded by const.$N$.

Now fix $n\geq m$. Then for $k=j2^{n-m},\cdots,(j+1)2^{n-m}-1$ we have,
using (\ref{eq8})
\[\begin{split}&u((k+1)2^{-n})-u(k2^{-n})=\int_{I_{kn}}\{f(W(t)+u(t))-
f(W(t))\}dt\\&=\int_{I_{kn}}\{f(W(t)+u_n(t))-f(W(t))\}dt+\sum_{l=n}^\infty
\int_{I_{kn}}\{f(W(t)+u_{l+1}(t))-f(W(t)+u_l(t))\}dt\end{split}\]
which we can write as
\zz\label{zu}u((k+1)2^{-n})-u(k2^{-n})=\sigma_{nk}(u(k2^{-n}))+\sum_{l=n}^\infty\sum_{r=k2^{l-n}}^{(k+1)2^{l-n}-1}
\rho_{l+1,2r+1}(u(2^{-l-1}(2r+1)),u(2^{-l}r))\z
from which we deduce
\zz\label{eq12}\sum_{k=j2^{n-m}}^{(j+1)2^{n-m}-1}|u((k+1)2^{-n})-u(k2^{-n})|\leq
\sum_{k=j2^{n-m}}^{(j+1)2^{n-m}-1}|\sigma_{nk}(u(k2^{-n}))|+\sum_{l=n}
^\infty\Omega_l\end{equation}
where $\Omega_l=\sum_{r=j2^{l-m}}^{(j+1)2^{l-m}-1}|\rho_{l+1,2r+1}(u(2^{-l-1}(2r+1)),u(2^{-l}r))|$.

We now proceed to estimate the two sums on the right of (\ref{eq12}), starting with the easier $\sigma_{nk}$ term. Using Lemma \ref{l3}
and the fact that $N<2^m$, we have $|\sigma_{nk}(x)|\leq C_2n^{1/2}2^{-n/2}(2^{-N}+|x|)$ and so
\zz\label{b1}\begin{split}\sum_{k=j2^{n-m}}^{(j+1)2^{n-m}-1}|\sigma_{nk}(u(k2^{-n}))|&\leq
C_2\sum_{k=j2^{n-m}}^{(j+1)2^{n-m}-1}n^{1/2}2^{-n/2}(2^{-N}+|u(k2^{-n})|)\\
&\leq C_3n^{1/2}2^{n/2-m}(\beta+2^{-m}N\alpha+2^{-N})\end{split}\z
using (\ref{eq11}).

Next we bound $\sum\Omega_l$, which we do in two stages. We first obtain a relatively crude bound by applying (\ref{eq7}) to each term,
and then obtain an improved by applying the crude bound together with Lemma (\ref{nlb}). To start with the crude bound, from (\ref{eq7})
we have $|\rho_{nk}(x,y)|\leq C_32^{-n/2}N^{1/2}(2^{-N}+|x-y|)$ and using this together with (\ref{eq10}) gives
\zz\label{sl}\Omega_l\leq C_42^{-l/2}N^{1/2}\{2^{-N}2^{l-m}+\alpha2^{-m}(l^{1/2}2^{l/2}+N)\}\z
and so
\zz\label{b2}\sum_{l=m}^N\Omega_l\leq C_5(N^{1/2}2^{-m-N/2}+\alpha2^{-m}N^2)\z
For $l>N$ we use $|u(t)-u(t')|\leq|t-t'|$ and (\ref{eq7}) to obtain
\zz\label{inf}\sum_{l=N+1}^\infty\Omega_l\leq\sum_{l=N+1}^\infty C_62^{l-m}l^{1/2}2^{-3l/2}\leq C_7N^{1/2}2^{-m-N/2}\z
and combining this with (\ref{b2}) we obtain
\zz\label{b3}\sum_{l=m}^\infty\Omega_l\leq C_8(N^{1/2}2^{-m-N/2}+\alpha2^{-m}N^2)\z

The second stage is to improve the estimate (\ref{b3}) by applying Lemma \ref{nlb} to obtain a better estimate for $\Omega_n$
for larger $n$; we use (\ref{b3}) to bound the $\gamma$ term in Lemma \ref{nlb}.

Let $N^{1/6}\leq n\leq N$. We define $\gamma_{nk}=u((k+1)2^{-n})-u(k2^{-n})-\sigma_{nk}(u(k2^{-n}))$, noting that (\ref{zu}) implies that
\zz\label{eqg}\sum_{k=j2^{n-m}}^{(j+1)2^{n-m}-1}|\gamma_{nk}|\leq\sum_{l=n}^\infty\Omega_l\leq C_8(N^{1/2}2^{-m-N/2}+\alpha2^{-m}N^2)\z
 Also we define
\[\Lambda_n=\sum_{k=j2^{n-m}}^{(j+1)2^{n-m}-1}|\rho_{n,k+1}(2^{-n}k,2^{-n}(k+1))\]
so that $\Omega_n\leq\Lambda_{n+1}$. Let $r=\lfloor2^{n/4}\rfloor$. In order to apply Lemma \ref{nlb} to estimate $\Lambda_n$, we will
split the sum into $r$-sized pieces. First we find $i\in\{0,1,\cdots,r-1\}$ such that, writing $s=\lfloor r^{-1}(2^{n-m}-i)\rfloor$, we
have $\sum_{t=0}^s|u(j2^{-m}+(i+tr)2^{-n})|\leq r^{-1}\psi_n$. Now we fix for the moment $t\in\{0,1,\cdots,s\}$ and apply Lemma
\ref{nlb} with $y_q=u((k+q)2^{-n})$ where $k=j2^{n-m}+i+tr$. We obtain
\[\sum_{q=1}^r|\rho_{n,k+q}(y_{q-1},y_q)|\leq C_9\left(2^{-3n/4}|u(k2^{-n})|+2^{-n/4}\sum_{q=0}^{r-1}|\gamma_{n,k+q}|
+2^{-2^{n/2}}\right)\]
 Summing over $t$ then gives
 \[\begin{split}\sum_{k=j2^{n-m}+i}^{(j+1)2^{n-m}-1}|\rho_{n,k+1}(2^{-n}k,2^{-n}(k+1))|\leq&C_92^{-3n/4}\sum_{t=0}^s
 |u(j2^{-m}+(i+tr)2^{-n})|\\&+C_9\left(+2^{-n/4}\sum_{k=j2^{n-m}+i}^{(j+1)2^{n-m}-1}|\gamma_{n,k}|+2^{n-2^{n/2}}\right)\end{split}\]
Also
\[\sum_{k=j2^{n-m}}^{j2^{n-m}+i-1}|\rho_{n,k+1}(2^{-n},2^{-n}(k+1))|\leq C_9\left(2^{-3n/4}|u(j2^{-m})|+2^{-n/4}\sum_{k=j2^{n-m}}^{j2^{n-m}
+i-1}|\gamma_{n,k}|+r2^{-2^{n/2}}\right)\]
From the last two inequalities, using (\ref{eq11}), (\ref{eqg}) and $|u(j2^{-m})|\leq\beta$, we find that
\[\Lambda_n\leq C_{10}\{2^{-m}(\beta+\alpha2^{-m}N)+2^{-m-n/4}(N^{1/2}2^{-N/2}+\alpha N^2)+2^{n-2^{n/2}}\}\]
Since $n\geq N^{1/6}$ the first term dominates so $\Lambda_n\leq C_{11}2^{-m}(\beta+\alpha2^{-m}N)$, and the same
bound holds for $\Omega_n\leq\Lambda_{n+1}$. We deduce that
\[\sum_{N^{1/6}\leq l\leq N}\Omega_l\leq C_{12}N2^{-m}(\beta+\alpha N2^{-m})\]
Using the original bound (\ref{sl}) for $l<N^{1/6}$ we have
\[\sum_{m\leq l<N^{1/6}}\Omega_l\leq C_{13}N^{1/2}\{2^{-N+N^{1/4}/2-m}+\alpha2^{-m}(N^{1/4}+2^{-m/2}N)\}\]
Combining these two estimates with (\ref{inf}) we get our improved bound.
\[\sum_{l=m}^\infty\Omega_l\leq C_{14}\{N2^{-m}(\beta+\alpha N2^{-m})+\alpha(2^{-m}N^{3/4}+2^{-3m/2}N^{3/2})\}\]
To conclude the proof we use this bound along with (\ref{b1}) in (\ref{eq12}) and obtain
\[\begin{split}\sum_{k=j2^{n-m}}^{(j+1)2^{n-m}-1}|u((k+1)2^{-n})-u(k2^{-n})|\leq&C_{15}(n^{1/2}2^{n/2-m}+N2^{-m})\\
&\times\{\beta+\alpha(N2^{-m}+N^{-1/4}+2^{-m/2}N^{1/2})\}\end{split}\]
for all $n$ with $m\leq n\leq N$. Comparing this with (\ref{eq10}) we see by the minimality of $\alpha$ that
\[\alpha\leq C_{15}\{\beta+\alpha(N2^{-m}+N^{-1/4}+2^{-m/2}N^{1/2})\}\]
Then if $m$ is large enough to ensure $C_{15}(N2^{-m}+N^{-1/4}+2^{-m/2}N^{1/2})<1/2$ it follows that $\alpha\leq2C_{15}\beta$. Then
applying (\ref{eq10}) with $n=m$ gives $|u((j+1)2^{-m})|\leq\beta+2C_{15}\beta(m^{1/2}2^{m/2}+N)2^{-m}\leq\beta(1+C_{16}N2^{-m})$
from which the required result follows.
\end{proof}

To complete the proof of Theorem \ref{mth}, using the notation of Lemma \ref{ll} let $m>m_0$ and $\beta_0=2^{-2^{3m/4}}$, and define
$\beta_j$ for $j=1,2,\cdots,2^m$ by the recurrence relation $\beta_{j+1}=\beta_j(1+K2^{-m}\log(1/\beta_j))$. Writing $\gamma_j
=\log(1/\beta_j)$ we then have
\[\gamma_{j+1}=\gamma_j-\log(1+K2^{-m}\gamma_j)\geq\gamma_j(1-K2^{-m})\]
so the sequence $(\gamma_j)$ is decreasing and
\[\gamma_j\geq\gamma_0(1-K2^{-m})^j\geq\gamma_0e^{-K-1}=2^{3m/4}e^{-K-1}\geq2^{2m/3}\]
for all $j=1,2,\cdots,2^m$, provided $m$ is large enough. Then for each $j$, $\beta_j$ is in the range specified in Lemma \ref{ll}, and
it follows from that lemma by induction on $j$ that $|u(j2^{-m})|\leq\beta_j$ for each $j$. Hence $|u(j2^{-m})|\leq2^{-2^{2m/3}}$ for
each $j$. This holds for all large enough $m$, and hence $u$ vanishes at all dyadic points in [0,1], and, as $u$ is continuous, $u=0$ on
[0,1]. This completes the proof of the theorem.

\section{An Application}\label{app}

We give an application of Theorem \ref{mth} to
convergence of Euler approximations to (\ref{eq1}) with variable step size.

In this section we assume $f$ is continuous and consider (\ref{eq1}) on a
bounded interval $[0,T]$. Given a partition
${\mathcal P}=\{0=t_0<t_1<\cdots<t_N=T\}$ of $[0,T]$ we consider the Euler
approximation to (\ref{eq1}) given by:
\[x_{n+1}=x_n+W(t_{n+1})-W(t_n)+(t_{n+1}-t_n)f(t_n,x_n)\]
for $n=0,\cdots,N-1$, with $x_0=0$. For such a partition ${\mathcal P}$ we
let $\delta({\mathcal P})=\max_{n=1}^N(t_{n}-t_{n-1})$. Then we have the
following:

\begin{cor}\label{cor}For almost every Brownian path $W$, for
any sequence\[
{\mathcal P}_k=\{t_0^{(k)},\cdots,t_{N_k}^{(k)}\}\]
of partitions with $\delta({\mathcal P}_k)\rightarrow0$, we have
\[\max_{n=1}^{N_k}|x_n^{(k)}-x(t_n^{(k)})|\rightarrow0\]
as $k\rightarrow\infty$, where $x(t)$ is the unique solution of
(\ref{eq1}) and $\{x_n^{(k)}\}$ is the Euler approximation using the partition
${\mathcal P}_k$.
\end{cor}

\begin{proof} Suppose $W$ is a path for which the conclusion of Theorem
\ref{mth} holds, and suppose there is a sequence of partitions with
$\delta({\mathcal P}_k)\rightarrow0$ such that
$\max_{n=1}^{N_k}|x_n^{(k)}-x(t_n^{(k)})|\geq\delta>0$. Then if we let
$u_n^{(k)}=x_n^{(k)}-W(t_n^{(k)})$ we have $|u_{n+1}^{(k)}-u_n^{(k)}|
\leq\|f\|_\infty(t_{n+1}^{(k)}-t_n^{(k)})$ so by Ascoli-Arzela, after
passing to a subsequence we have a continuous $u$ on $[0,T]$ such that 
$\max_{n=1}^{N_k}|u_n^{(k)}-u(t_n^{(k)})|\rightarrow0$. Then writing
$y(t)=u(t)+W(t)$ we see that $y\neq x$ and, using the continuity of $f$,
that $y$ satisfies (\ref{eq1}), contradicting the conclusion of the
theorem. Corollary \ref{cor} is proved.
\end{proof}

The point of Corollary \ref{cor} is that the partitions can be chosen
arbitrarily, no `non-anticipating' condition is required. For  general
SDE's with non-additive noise and sufficiently smooth coefficients Euler
approximations will converge to the solution provided the partition
points $t_n$ are stopping times, but this condition is rather
restrictive for numerical practice, and an example is given in section
4.1 of \cite{gl} of a natural variable step-size Euler scheme for a simple SDE
which converges to the wrong limit. \cite{gl} also contains related results
and discussion.\vspace{.2cm}

{\bf Acknowledgement.} The author is grateful to Istvan Gy\"{o}ngy for
drawing his attention to Krylov's question and for valuable discussions.

\bibliographystyle{amsplain}

\end{document}